\documentclass{article}%
\usepackage{amsmath}
\usepackage{amsfonts}
\usepackage{amssymb}
\usepackage{graphicx}%
\providecommand{\U}[1]{\protect \rule{.1in}{.1in}}
\newtheorem{theorem}{Theorem}

\newtheorem{definition}[theorem]{Definition}
\newtheorem{example}[theorem]{Example}

\newtheorem{lemma}[theorem]{Lemma}

\newtheorem{problem}[theorem]{Problem}
\newtheorem{proposition}[theorem]{Proposition}
\newtheorem{remark}[theorem]{Remark}

\newenvironment{proof}[1][Proof]{\noindent \textbf{#1.} }{\  $\Box$}
\begin{document}

\title{The Independence under Sublinear Expectations}
\author{Mingshang HU\\School of Mathematics\\Shandong
University\\250100, Jinan, China\\humingshang@sdu.edu.cn}

\maketitle

\begin{abstract}
We show that, for two non-trivial random variables $X$ and $Y$ under
a sublinear expectation space, if $X$ is independent from $Y$ and
$Y$ is independent from $X$, then $X$ and $Y$ must be maximally
distributed.
\end{abstract}

\section{Introduction}
Peng \cite{Peng2004,Peng2005,Peng2006a,Peng2006b} introduced the
important notions of distributions and independence under the
sublinear expectation framework. Like classical linear expectations,
the independence play a key role in the sublinear analysis.

Unfortunately, $Y$ is independent from $X$ does not imply that $X$
is independent from $Y$. But if $X$ and $Y$ are maximally
distributed, this holds true. A natural problem is whether the
maximal distribution is the only distribution? In this paper, we
give an affirmative answer to this problem.

This paper is organized as follows: in Section 2, we recall some
basic results of sublinear expectations. The main result is given
and proved in Section 3.

\section{Basic settings}

We present some preliminaries in the theory of sublinear
expectations. More details of this section can be found in [7-14].

Let $\Omega$ be a given set and let $\mathcal{H}$ be a linear space of real
functions defined on $\Omega$ such that $c\in \mathcal{H}$ for all constants
$c$ and $|X|\in \mathcal{H}$ if $X\in \mathcal{H}$. We further suppose that if
$X_{1},\ldots,X_{n}\in \mathcal{H}$, then $\varphi(X_{1},\cdots,X_{n}%
)\in \mathcal{H}$ for each $\varphi \in C_{b.Lip}(\mathbb{R}^{n})$, where
$C_{b.Lip}(\mathbb{R}^{n})$ denotes the space of bounded and Lipschitz
functions. $\mathcal{H}$ is considered as the space of random variables.

\begin{definition}
A sublinear expectation $\mathbb{\hat{E}}$ on $\mathcal{H}$ is a functional
$\mathbb{\hat{E}}:\mathcal{H}\rightarrow \mathbb{R}$ satisfying the following
properties: for all $X,$ $Y\in \mathcal{H}$, we have

\begin{description}
\item[(a)] Monotonicity: $\mathbb{\hat{E}}[X]\geq \mathbb{\hat{E}}[Y]$ if
$X\geq Y$.

\item[(b)] Constant preserving: $\mathbb{\hat{E}}[c]=c$ for $c\in \mathbb{R}$.

\item[(c)] Sub-additivity: $\mathbb{\hat{E}}[X+Y]\leq \mathbb{\hat{E}%
}[X]+\mathbb{\hat{E}}[Y]$.

\item[(d)] Positive homogeneity: $\mathbb{\hat{E}}[\lambda X]=\lambda
\mathbb{\hat{E}}[X]$ for $\lambda \geq0$.
\end{description}

The triple $(\Omega,\mathcal{H},\mathbb{\hat{E}})$ is called a sublinear
expectation space (compare with a probability space $(\Omega,\mathcal{F},P)$).
\end{definition}

\begin{remark}
If the inequality in (c) is equality, then $\mathbb{\hat{E}}$ is a linear
expectation on $\mathcal{H}$. We recall that the notion of the above sublinear
expectations was systematically introduced by Artzner, Delbaen, Eber and Heath
\cite{ADEH1, ADEH2}, in the case where $\Omega$ is a finite set, and by
Delbaen \cite{Delb} for the general situation with the notation of risk
measure: $\rho(X):=\mathbb{\hat{E}}[-X]$. See also Huber \cite{Huber} for even
earlier study of this notion $\mathbb{\hat{E}}$ (called the upper expectation
$\mathbf{E}^{\ast}$ in Ch. 10 of \cite{Huber}).
\end{remark}

\begin{remark}
It is easy to deduce from (d) that
\[
\mathbb{\hat{E}}[\lambda X]=\lambda^{+} \mathbb{\hat{E}}[X]+\lambda^{-}
\mathbb{\hat{E}}[-X] \  \  \text{for}\  \lambda \in \mathbb{R}.
\]

\end{remark}

\begin{remark}
Let $\{ E_{\theta}:\theta \in \Theta \}$ be a family of linear expectations
defined on $\mathcal{H}$. Then
\[
\mathbb{\hat{E}}[X]:=\sup_{\theta \in \Theta}E_{\theta}[X]\  \  \text{for}%
\ X\in \mathcal{H}
\]
is a sublinear expectation. In fact, every sublinear expectation has
this kind of representation (see Peng \cite{Peng2007, Peng2008}).
\end{remark}

Let $X=(X_{1},\ldots,X_{n})$, $X_{i}\in \mathcal{H}$, denoted by $X\in
\mathcal{H}^{n}$, be a given $n$-dimensional random vector on a sublinear
expectation space $(\Omega,\mathcal{H},\mathbb{\hat{E}})$. We define a
functional on $C_{b.Lip}(\mathbb{R}^{n})$ by
\[
\mathbb{\hat{F}}_{X}[\varphi]:=\mathbb{\hat{E}}[\varphi(X)]\text{ \ for all
}\varphi \in C_{b.Lip}(\mathbb{R}^{n}).
\]
The triple $(\mathbb{R}^{n},C_{b.Lip}(\mathbb{R}^{n}),\mathbb{\hat{F}}%
_{X}[\cdot])$ forms a sublinear expectation space. $\mathbb{\hat{F}}_{X}$ is
called the distribution of $X$.

\begin{definition}
A random vector $X\in \mathcal{H}^{n}$ is said to have distributional
uncertainty if the distribution $\mathbb{\hat{F}}_{X}$ is not a linear expectation.
\end{definition}

The following simple property is very useful in sublinear analysis.

\begin{proposition}
\label{pro1}Let $X,$ $Y\in \mathcal{H}$ be such that $\hat{\mathbb{E}}%
[Y]=-\hat{\mathbb{E}}[-Y]$. Then we have%
\[
\hat{\mathbb{E}}[X+Y]=\hat{\mathbb{E}}[X]+\hat{\mathbb{E}}[Y].
\]
In particular, if $\hat{\mathbb{E}}[Y]=\hat{\mathbb{E}}[-Y]=0$, then
$\hat{\mathbb{E}}[X+Y]=\hat{\mathbb{E}}[X]$.
\end{proposition}

\begin{proof}
It is simply because we have $\hat{\mathbb{E}}[X+Y]\leq \hat{\mathbb{E}%
}[X]+\hat{\mathbb{E}}[Y]$ and%
\[
\hat{\mathbb{E}}[X+Y]\geq \hat{\mathbb{E}}[X]-\hat{\mathbb{E}}[-Y]=\hat
{\mathbb{E}}[X]+\hat{\mathbb{E}}[Y].
\]

\end{proof}

Noting that $\hat{\mathbb{E}}[c]=-\hat{\mathbb{E}}[-c]=c$ for all
$c\in \mathbb{R}$, we immediately have
\[
\hat{\mathbb{E}}[X+c]=\hat{\mathbb{E} }[X]+c.
\]

The following notion of independence plays an important role in the sublinear
expectation theory.

\begin{definition}
Let $(\Omega,\mathcal{H},\mathbb{\hat{E}})$ be a sublinear expectation space.
A random vector $Y=(Y_{1},\cdots,Y_{n})\in \mathcal{H}^{n}$ is said to be
independent from another random vector $X=(X_{1},\cdots,X_{m})\in
\mathcal{H}^{m}$ under $\mathbb{\hat{E}}[\cdot]$ if for each test function
$\varphi \in C_{b.Lip}(\mathbb{R}^{m}\times \mathbb{R}^{n})$ we have
\[
\mathbb{\hat{E}}[\varphi(X,Y)]=\mathbb{\hat{E}}[\mathbb{\hat{E}}%
[\varphi(x,Y)]_{x=X}].
\]

\end{definition}

\begin{remark}
Under a sublinear expectation space, $Y$ is independent from $X$ means that
the distributional uncertainty of $Y$ does not change after the realization of
$X=x$. Or, in other words, the \textquotedblleft conditional sublinear
expectation\textquotedblright \ of $Y$ knowing $X$ is $\mathbb{\hat{E}}%
[\varphi(x,Y)]_{x=X}$. In the case of linear expectation, this notion of
independence is just the classical one.
\end{remark}

It is important to note that under sublinear expectations the condition
\textquotedblleft$Y$ is independent from $X$\textquotedblright \ does not imply
automatically that \textquotedblleft$X$ is independent from $Y$%
\textquotedblright. See the following example:

\begin{example}
We consider a case where $\mathbb{\hat{E}}$ is a sublinear expectation and
$X,Y\in \mathcal{H}$ are identically distributed with $\mathbb{\hat{E}%
}[X]=\mathbb{\hat{E}}[-X]=0$ and $\bar{\sigma}^{2}=\mathbb{\hat{E}}%
[X^{2}]>\underline{\sigma}^{2}=-\mathbb{\hat{E}}[-X^{2}]$. We also assume that
$\mathbb{\hat{E}}[|X|]=\mathbb{\hat{E}}[X^{+}+X^{-}]>0$, thus $\mathbb{\hat
{E}}[X^{+}]=\frac{1}{2}\mathbb{\hat{E}}[|X|+X]=$$\frac{1}{2}\mathbb{\hat{E}%
}[|X|]>0$. In the case where $Y$ is independent from $X$, we have%
\[
\mathbb{\hat{E}}[XY^{2}]=\mathbb{\hat{E}}[X^{+}\bar{\sigma}^{2}-X^{-}%
\underline{\sigma}^{2}]=(\bar{\sigma}^{2}-\underline{\sigma}^{2}%
)\mathbb{\hat{E}}[X^{+}]>0.
\]
But if $X$ is independent from $Y$ we have%
\[
\mathbb{\hat{E}}[XY^{2}]=0.
\]

\end{example}

The following is a representation theorem of the distribution of a
random vector (see \cite{DHP,H-P,Peng2010}).

\begin{theorem}
\label{th0}Let $X\in \mathcal{H}^{n}$ be a $n$-dimensional random vector on a
sublinear expectation space $(\Omega,\mathcal{H},\mathbb{\hat{E}})$. Then
there exists a weakly compact family of probability measures $\mathcal{P}$ on
$(\mathbb{R}^{n},\mathcal{B}(\mathbb{R}^{n}))$ such that%
\[
\mathbb{\hat{F}}_{X}[\varphi]=\mathbb{\hat{E}}[\varphi(X)]=\max_{P\in
\mathcal{P}}E_{P}[\varphi]\text{ \ for all }\varphi \in C_{b.Lip}%
(\mathbb{R}^{n}).
\]

\end{theorem}

\begin{definition}
A $n$-dimensional random vector $X\in \mathcal{H}^{n}$ on a sublinear
expectation space $(\Omega,\mathcal{H},\mathbb{\hat{E}})$ is called maximally
distributed if there exists a closed set $\Gamma \subset \mathbb{R}^{n}$ such
that%
\[
\mathbb{\hat{F}}_{X}[\varphi]=\mathbb{\hat{E}}[\varphi(X)]=\sup_{x\in \Gamma
}\varphi(x)\text{ \ for all }\varphi \in C_{b.Lip}(\mathbb{R}^{n}).
\]

\end{definition}

\begin{remark}
In Peng \cite{Peng2007, Peng2008}, the definition of maximal
distribution demands the convexity of $\Gamma$. Here, we still call
it the maximal distribution without the convexity of $\Gamma$ for
convenience.
\end{remark}

\section{Main result}

We now discuss some cases under which $X$ is independent from $Y$ and $Y$ is
independent from $X$. In this section, we do not consider the following two
trivial cases:

\begin{description}
\item[(i)] The distributions of $X$ and $Y$ are linear;

\item[(ii)] At least one of $X$ and $Y$ is constant.
\end{description}

The following example is a non-trivial case.

\begin{example}
Let $\Omega=\mathbb{R}^{2}$, $\mathcal{H}=C_{b.Lip}(\mathbb{R}^{2})$ and let
$K_{1}$ and $K_{2}$ be two closed sets in $\mathbb{R}$. We define%
\[
\mathbb{\hat{E}}[\varphi]=\sup_{(x,y)\in K_{1}\times K_{2}}\varphi(x,y)\text{
\ for all }\varphi \in C_{b.Lip}(\mathbb{R}^{2}).
\]
It is easy to check that $\xi(x,y):=x$ is independent from $\eta(x,y):=y$ and
$\eta$ is independent from $\xi$.
\end{example}

We will prove that this is the only case. The following theorem is the main
theorem in this section.

\begin{theorem}
\label{th1}Suppose that $X\in \mathcal{H}$ has distributional uncertainty and
$Y$ $\in \mathcal{H}$ is not a constant on a sublinear expectation space
$(\Omega,\mathcal{H},\mathbb{\hat{E}})$. If $X$ is independent from $Y$ and
$Y$ is independent from $X$, then $X$ and $Y$ must be maximally distributed.
\end{theorem}

In order to prove this theorem, we need the following lemmas.

\begin{lemma}
\label{le1}Suppose $X\in \mathcal{H}$ has distributional uncertainty on a
sublinear expectation space $(\Omega,\mathcal{H},\mathbb{\hat{E}})$. Then
there exists a $\varphi \geq0$ such that $\mathbb{\hat{E}}[\varphi(X)]=1$ and
$-\mathbb{\hat{E}}[-\varphi(X)]<1$.
\end{lemma}

\begin{proof}
We first claim that there exists a $\varphi_{0}\geq0$ such that $-\mathbb{\hat
{E}}[-\varphi_{0}(X)]<\mathbb{\hat{E}}[\varphi_{0}(X)]$. Otherwise, for each
$\varphi \geq0$, we have $\mathbb{\hat{E}}[\varphi(X)]=-\mathbb{\hat{E}%
}[-\varphi(X)]$. For each $\varphi \in C_{b.Lip}(\mathbb{R})$, let $M:=\inf \{
\varphi(x):x\in \mathbb{R}\}$, then $M+\varphi \geq0$ and%
\[
\mathbb{\hat{E}}[\varphi(X)]+M=\mathbb{\hat{E}}[\varphi(X)+M]=-\mathbb{\hat
{E}}[-\varphi(X)-M]=-\mathbb{\hat{E}}[-\varphi(X)]+M,
\]
which implies that $\mathbb{\hat{E}}[\varphi(X)]=-\mathbb{\hat{E}}%
[-\varphi(X)]$ for each $\varphi \in C_{b.Lip}(\mathbb{R})$. It follows from
Proposition \ref{pro1} that%
\[
\mathbb{\hat{E}}[\varphi(X)+\psi(X)]=\mathbb{\hat{E}}[\varphi(X)]+\mathbb{\hat
{E}}[\psi(X)]\  \  \text{for each }\varphi,\psi \in C_{b.Lip}(\mathbb{R}),
\]
which contradics our assumption. We then take $\varphi^{\ast}=(\mathbb{\hat
{E}}[\varphi_{0}(X)])^{-1}\varphi_{0}\geq0$. It is easy to verify that
$\mathbb{\hat{E}}[\varphi^{\ast}(X)]=1$ and $-\mathbb{\hat{E}}[-\varphi^{\ast
}(X)]<1$, the proof is complete.
\end{proof}

\begin{lemma}
\label{le2}Suppose $X$ and $Y$ are as in Theorem \ref{th1}. If $X$ is
independent from $Y$ and $Y$ is independent from $X$, then we have%
\[
\mathbb{\hat{E}}[(\psi(Y)-\mathbb{\hat{E}}[\psi(Y)])^{+}]=0\  \  \text{for all
}\psi \in C_{b.Lip}(\mathbb{R}).
\]

\end{lemma}

\begin{proof}
It follows from Lemma \ref{le1} that there exists a $\varphi^{\ast}\geq0$ such
that $\mathbb{\hat{E}}[\varphi^{\ast}(X)]=1$ and $-\mathbb{\hat{E}}%
[-\varphi^{\ast}(X)]<1$. We set $\varepsilon=-\mathbb{\hat{E}}[-\varphi^{\ast
}(X)]\in \lbrack0,1)$ and define%
\[
G(a)=\mathbb{\hat{E}}[a\varphi^{\ast}(X)]=a^{+}\mathbb{\hat{E}}[\varphi^{\ast
}(X)]+a^{-}\mathbb{\hat{E}}[-\varphi^{\ast}(X)]=a^{+}-\varepsilon
a^{-}\  \  \text{for }a\in \mathbb{R}.
\]
Note that $Y$ is independent from $X$, then we have%
\begin{equation}
\mathbb{\hat{E}}[\varphi^{\ast}(X)\psi(Y)]=\mathbb{\hat{E}}[\mathbb{\hat{E}%
}[\psi(Y)]\varphi^{\ast}(X)]=G(\mathbb{\hat{E}}[\psi(Y)])\  \  \text{for all
}\psi \in C_{b.Lip}(\mathbb{R}). \label{eq1}%
\end{equation}
On the other hand, $X$ is independent from $Y$, then we get%
\begin{equation}
\mathbb{\hat{E}}[\varphi^{\ast}(X)\psi(Y)]=\mathbb{\hat{E}}[\mathbb{\hat{E}%
}[\psi(y)\varphi^{\ast}(X)]_{y=Y}]=\mathbb{\hat{E}}[G(\psi(Y))]\  \  \text{for
all }\psi \in C_{b.Lip}(\mathbb{R}). \label{eq2}%
\end{equation}
Combining (\ref{eq2}) with (\ref{eq1}), we obtain%
\begin{equation}
\mathbb{\hat{E}}[G(\psi(Y))]=G(\mathbb{\hat{E}}[\psi(Y)])\  \  \text{for all
}\psi \in C_{b.Lip}(\mathbb{R}). \label{eq3}%
\end{equation}
Noting that $G\circ \psi \in C_{b.Lip}(\mathbb{R})$ for each $\psi \in
C_{b.Lip}(\mathbb{R})$, applying equation (\ref{eq3}) to $G\circ \psi$, we have%
\[
\mathbb{\hat{E}}[G\circ G(\psi(Y))]=G\circ G(\mathbb{\hat{E}}[\psi
(Y)])\  \  \text{for all }\psi \in C_{b.Lip}(\mathbb{R}).
\]
Denote
\[
G^{\circ n}=\underset{n}{\underbrace{G\circ G\circ \cdots \circ G}},
\]
continuing the above process, we can get%
\begin{equation}
\mathbb{\hat{E}}[G^{\circ n}(\psi(Y))]=G^{\circ n}(\mathbb{\hat{E}}%
[\psi(Y)])\  \  \text{for all }\psi \in C_{b.Lip}(\mathbb{R}). \label{eq4}%
\end{equation}
It is easy to check that $G^{\circ n}(a)=a^{+}-\varepsilon^{n}a^{-}$. By%
\[
|\mathbb{\hat{E}}[G^{\circ n}(\psi(Y))]-\mathbb{\hat{E}}[\psi^{+}%
(Y)]|=|\mathbb{\hat{E}}[\psi^{+}(Y)-\varepsilon^{n}\psi^{-}(Y)]-\mathbb{\hat
{E}}[\psi^{+}(Y)]|\leq \varepsilon^{n}\mathbb{\hat{E}}[\psi^{-}(Y)]
\]
and $G^{\circ n}(\mathbb{\hat{E}}[\psi(Y)])\ =(\mathbb{\hat{E}}[\psi
(Y)])^{+}-\varepsilon^{n}(\mathbb{\hat{E}}[\psi(Y)])^{-}$, we can deduce by
letting $n\rightarrow \infty$ that%
\begin{equation}
\mathbb{\hat{E}}[\psi^{+}(Y)]=(\mathbb{\hat{E}}[\psi(Y)])^{+}\  \  \text{for all
}\psi \in C_{b.Lip}(\mathbb{R}). \label{eq5}%
\end{equation}
For each $\psi \in C_{b.Lip}(\mathbb{R})$, applying equation (\ref{eq5}) to
$\tilde{\psi}:=\psi-\mathbb{\hat{E}}[\psi(Y)]$, we obtain the result. The
proof is complete.
\end{proof}

\textbf{Proof of Theorem \ref{th1}.} It follows from Theorem \ref{th0} that
there exists a weakly compact family of probability measures $\mathcal{P}$ on
$(\mathbb{R},\mathcal{B}(\mathbb{R}))$ such that%
\begin{equation}
\mathbb{\hat{F}}_{Y}[\psi]=\mathbb{\hat{E}}[\psi(Y)]=\max_{P\in \mathcal{P}%
}E_{P}[\psi]\text{ \ for all }\psi \in C_{b.Lip}(\mathbb{R}). \label{eq6}%
\end{equation}
For this $\mathcal{P}$, we set%
\begin{equation}
c(A):=\sup_{P\in \mathcal{P}}P(A)\  \  \text{for all }A\in \mathcal{B}%
(\mathbb{R}). \label{eq7}%
\end{equation}
By Lemma \ref{le2} and (\ref{eq6}), we have%
\begin{equation}
\mathbb{\hat{E}}[(\psi(Y)-\mathbb{\hat{E}}[\psi(Y)])^{+}]=\max_{P\in
\mathcal{P}}E_{P}[(\psi-\mathbb{\hat{E}}[\psi(Y)])^{+}]=0\  \  \text{for all
}\psi \in C_{b.Lip}(\mathbb{R}). \label{eq8}%
\end{equation}
From this, it is easy to obtain that $c(\{y:\psi(y)>\mathbb{\hat{E}}%
[\psi(Y)]\})=0$ for each $\psi \in C_{b.Lip}(\mathbb{R})$. For each given
$\psi_{0}\in C_{b.Lip}(\mathbb{R})$, we set%
\[
A:=\{y\in \mathbb{R}:\psi_{0}(y)=\mathbb{\hat{E}}[\psi_{0}(Y)]\}.
\]
It is easy to verify that $A$ is a closed set. We first assert that $c(A)>0$.
Otherwise,
\begin{equation}
c(\{y:\psi_{0}(y)\geq \mathbb{\hat{E}}[\psi_{0}(Y)]\})\leq c(\{y:\psi
_{0}(y)>\mathbb{\hat{E}}[\psi_{0}(Y)]\})+c(A)=0, \label{eq9}%
\end{equation}
by (\ref{eq6}) and (\ref{eq9}), we get%
\[
\mathbb{\hat{E}}[\psi_{0}(Y)]=\max_{P\in \mathcal{P}}E_{P}[\psi_{0}%
]<\mathbb{\hat{E}}[\psi_{0}(Y)],
\]
this is a contradiction, thus $c(A)>0$. We then claim that there exists a
$y_{0}\in A$ such that
\[
\psi(y_{0})\leq \mathbb{\hat{E}}[\psi(Y)]\text{ \ for all }\psi \in
C_{b.Lip}(\mathbb{R}).
\]
Otherwise, for each $\tilde{y}\in A$, there exists a $\tilde{\psi}\in
C_{b.Lip}(\mathbb{R})$ such that $\tilde{\psi}(\tilde{y})>\mathbb{\hat{E}%
}[\tilde{\psi}(Y)]$. Note that $c(\{y:\tilde{\psi}(y)>\mathbb{\hat{E}}%
[\tilde{\psi}(Y)]\})=0$, then there exists a $\tilde{\varepsilon}>0$ such that
$c([\tilde{y}-\tilde{\varepsilon},\tilde{y}+\tilde{\varepsilon}])=0$. Noting
that $A$ is closed, by the Heine-Borel theorem, there exists a sequence
$\{(y_{n},\varepsilon_{n}):n=1,2,\cdots \}$ such that
\[
A\subset \cup_{n}[y_{n}-\varepsilon_{n},y_{n}+\varepsilon_{n}]\  \text{and
}c([y_{n}-\varepsilon_{n},y_{n}+\varepsilon_{n}])=0.
\]
Thus, $c(A)\leq \sum_{n=1}^{\infty}c([y_{n}-\varepsilon_{n},y_{n}%
+\varepsilon_{n}])=0$, which contradicts to $c(A)>0$. Take $B=cl(\{y_{0}%
:\psi_{0}\in C_{b.Lip}(\mathbb{R})\})$ and $\mathcal{P}^{\prime}=\{ \delta
_{y}:y\in B\}$, then
\[
\mathbb{\hat{F}}_{Y}[\psi]=\mathbb{\hat{E}}[\psi(Y)]=\max_{P\in \mathcal{P}%
^{\prime}}E_{P}[\psi]\text{ \ for all }\psi \in C_{b.Lip}(\mathbb{R}),
\]
which implies that $Y$ is maximally distributed. Similarly, we can prove that
$X$ is maximally distributed. The proof is complete.

\bigskip

\begin{remark}
It is easy to check that $(X,Y)$ is maximally distributed. Since
$Y=(Y_{1},\ldots,Y_{m})\in \mathcal{H}^{m}$ independent from $X=(X_{1}%
,\ldots,X_{n})\in \mathcal{H}^{n}$ implies $Y_{i}$ independent from $X_{j}$ for
$i\leq m$ and $j\leq n$, the result of Theorem \ref{th1} still holds.
\end{remark}

\begin{definition}
Let $(\Omega,\mathcal{H},\mathbb{\hat{E}})$ be a sublinear expectation space.
A random vector $Y\in \mathcal{H}^{n}$ is said to be weakly independent from
another random vector $X\in \mathcal{H}^{m}$ under $\mathbb{\hat{E}}[\cdot]$
if
\[
\mathbb{\hat{E}}[\varphi(X)\psi(Y)]=\mathbb{\hat{E}}[\mathbb{\hat{E}}%
[\varphi(x)\psi(Y)]_{x=X}]\  \  \text{for each }\varphi,\psi \in C_{b.Lip}%
(\mathbb{R}).
\]

\end{definition}

\begin{remark}
It is easy to see from the proof that the result of Theorem \ref{th1} still
holds under weak independence.
\end{remark}

\begin{problem}
Whether weak independence is independence? Moreover, what kind of sets can
determine sublinear expectations? Whether $\mathcal{H}_{0}:=\{ \varphi
(x)\psi(y):\varphi,\psi \in C_{b.Lip}(\mathbb{R})\}$ is enough to determine
sublinear expectations?
\end{problem}

\end{document}